\documentclass[10pt]{article}
\usepackage{amssymb,amsthm,amsmath}
\usepackage{color,graphicx}
\usepackage{authblk}
\usepackage[dvipsnames]{xcolor}
\usepackage{tikz,cancel}
\usepackage{mathrsfs}
\usetikzlibrary{patterns}
\usepackage{mathtools}
\usepackage{caption}
\usepackage{lineno}
\usepackage{hyperref}
 \hypersetup{linkcolor=blue, colorlinks=true,citecolor = red}
\usepackage[capitalise]{cleveref}
\usepackage{upref}
\usepackage{changes}
\usepackage{mathrsfs}

\newtheorem{corollary}{Corollary}[section]
\newtheorem{definition}[corollary]{Definition}

\newtheorem{proposition}[corollary]{Proposition}
\newtheorem{remark}[corollary]{Remark}
\newtheorem{theorem}[corollary]{Theorem}
\numberwithin{equation}{section}

\renewcommand{\leq}{\leqslant}

\renewcommand{\geq}{\geqslant}

\makeatletter
\def\section{\@startsection {section}{1}{\z@}{-3.5ex plus -1ex minus
    -.2ex}{2.3ex plus .2ex}{\large\bf}}

\makeatother

\begin{document}

\title{\bf Global Exponential Stabilization for a Simplified Fluid-Particle Interaction System }
\author[1]{Marius Tucsnak}
\affil[1]{Institut de Math\'{e}matiques de Bordeaux UMR 5251, Universit\'{e} de Bordeaux/Bordeaux INP/CNRS, 351 Cours de la Lib\'{e}ration, 33 405 TALENCE, France} 
\author[2]{Zhuo Xu}
\affil[2]{Institut de Math\'{e}matiques de Bordeaux UMR 5251, Universit\'{e} de Bordeaux/Bordeaux INP/CNRS, 351 Cours de la Lib\'{e}ration, 33 405 TALENCE, France}
\date{}

\maketitle

\begin{abstract}
\vskip 0.2in
\par
This work considers a system coupling a viscous Burgers equation (aimed to describe a simplified model of $1D$ fluid flow) with
the ODE describing the motion of a point mass moving inside the fluid. The point mass is possibly under the action of a feedback control. Our main contributions are that we prove two global exponential stability results. More precisely, we first show that the velocity field corresponding to the free dynamics case is globally exponentially stable. We next show that, in the presence of the feedback control both the velocity field and the distance from the mass point to a prescribed target position decay exponentially. The proofs of these results heavily rely on the use of a special test function allowing both to prove that the mass point stays away from the boundary and to construct a perturbed Lyapunov function.



\end{abstract}

\tableofcontents
\section{Introduction}\label{sec_very_first}
\pagenumbering{arabic}\setcounter{page}{1}
\vspace{0.5cm}
In recent years, the equations that describe fluid-structure interactions  have attracted significant attention in mathematics and physics. Early references such as Conca, San Mart{\'i}n and Tucsnak  \cite{Tucsnak20001019}, Desjardins and Esteban \cite{Desjardins199959}, Tucsnak \cite{tucsnak2002global}, Takahashi \cite{Takahashi2003analysis}, Vázquez and Zuazua \cite{vazquez2003large}, \cite{Vzquez2006LACKOC} provided  well-posedness results for such interaction systems.

There  exist also results on the stabilization and the controllability
of such type of systems. These results typically require smooth and close to equilibrium initial data. One persistent difficulty in obtaining global results is to ensure the fact that there are no
collisions of the solid with the boundary. This can be granted  by the choice of the data of the problem (see, for instance, Hesla \cite{Helsa2004PhD75H}, Hillairet \cite{hillairet2007lack}, Hillairet and G\'erard Varet \cite{gerard2010regularity} or Hillairet and Takahashi
\cite{hillairet2010blow}) and/or by the design of the control mechanism,  Takahashi, Tucsnak, and Weiss \cite{Takahashi2015stabilization}.

In this paper, we consider a one-dimensional model first introduced by \cite{vazquez2003large}, \cite{Vzquez2006LACKOC}, where  the existence, uniqueness and large-time behavior of the solutions have been first investigated, in the case of a fluid filling the entire space. Subsequently, Doubova and Fernandez-Cara \cite{doubova2005some} addressed the boundary controllability problem for this system, proving local null controllability for the nonlinear system when control acts on both boundaries. Then, specifically regarding the establishment of null controllability when control is applied at only one end of the domain, Liu, Takahashi, and Tucsnak \cite{liu2013single} solved this problem by employing spectral methods and introducing a the so called {\em source point method}.

In 2015, Cîndea, Micu, Roventa, and Tucsnak \cite{tucsnak2015particle} considered a model in which control is applied directly to the particle. Their analysis focused on a closed-loop system with  initial data close to equilibrium, and they examined the  long-time behavior of the solutions. Using a combination of spectral and fixed-point techniques, they established local null controllability for the nonlinear interaction system.

There are two questions left open in  \cite{tucsnak2015particle}. The first one is the global well-posedness for the closed loop interaction system. The second consists in obtaining decay rates in the presence of the considered feedback control.

In this paper  we consider the  system:
\begin{align}
	\label{conversition_of_mom_of_solid_without_control}	&\dot{v}(t,y) + v(t,y) v_y(t,y)- v_{yy}(t,y)  = 0 & t &\in (0,T), \quad y \in (-1,1), \quad y \neq h(t), \\
	\label{boundary_condition_of_fluid_without_control}	&v(t,-1) = v(t,1) = 0 & t &\in (0,T), \\
	\label{campability_condition_without_control}&\dot{h}(t) = v(t,h(t)) & t &\in (0,T), \\
	\label{newton_second_law_without_control}&\ddot{h}(t) = [v_y](t,h(t))
	+u(t)  & t &\in (0,T), \\
	\label{inlitial_date_of_veloctiy_without_control}&v(0,y) = v_0(y) & y &\in (-1,1),\\
	\label{inlitial_date_of_solid_without_control}&h(0)=h_0,\qquad \qquad \dot{h}(0)=g_0.
\end{align}
where
\begin{equation}\label{definition_of_jump}
\left[v_y\right](t,h(t))=v_y\left(t,h(t)^+\right)-v_y\left(t,h(t)^-\right)\qquad \qquad (t\geq 0).
\end{equation}
Here, we consider a fluid velocity field denoted by $v = v(t, y)$, which is governed by the viscous Burgers equation. Additionally, the motion of a particle within this field obeys Newton's second law (as referenced in \eqref{newton_second_law_without_control}). The position of the particle at any given time is denoted by $h(t)$, and its velocity is represented by $\dot{h}(t)$, both of which are functions of time only. The input force acting on the particle, denoted by $u$, also depends solely on time $t$.

We observe that the system defined by  \eqref{conversition_of_mom_of_solid_without_control}- \eqref{inlitial_date_of_solid_without_control} constitutes a free boundary value problem since the position $h(t)$ of the particle, thus the domain filled by the fluid, is one of the unknowns of the problem. This is a property  which \eqref{conversition_of_mom_of_solid_without_control}- \eqref{inlitial_date_of_solid_without_control} shares with Stefan problems, see Koga and Krstic \cite{koga2022control} for the control theoretic properties of the corresponding systems. However, in our case, the presence of a second derivative of $h$ in \eqref{newton_second_law_without_control}, which is the main difference with respect to Stefan type systems (where a single time derivative appears) yields more regularity of the state trajectories.
The presence of a free boundary necessitates a precise definition of the concept of a finite energy solution. Such a definition has been meticulously formulated in \cite[Definition 1.1]{tucsnak2015particle}.
\begin{definition}\label{definition_of_finite_energy}
	Given $T>0$, $v_0\in L^2(-1,1)$, $g_0\in \mathbb{R}$, $h_0\in (-1,1)$ and $u\in L^2(0,T)$, we say that
	\begin{equation}
		\begin{bmatrix}
			v\\
			g\\
			h
		\end{bmatrix}\in \left[ C\left([0,T];\;L^2(-1,1)\right)\cap  L^2\left((0,T);\;H^1_0(-1,1)\right)\right]\times L^2(0,T)\times H^1(0,T),
	\end{equation}
	is a finite energy solution of system \eqref{conversition_of_mom_of_solid_without_control}-\eqref{inlitial_date_of_solid_without_control} on $[0,T]$ if $h(0)=h_0$, $\dot{h}(t)=g(t)=v\left(t,h(t)\right)$, $h(t)\in (-1,1)$ for almost every $t\in [0,T]$ and
	\begin{align}
		&\int_{-1}^{1}v(t,y)\psi(t,y)\,{\rm d}y-\int_{-1}^{1}v_0(y)\psi(0,y)\,{\rm d}y+g(t)l(t)-g_0l(0)-\int_{0}^{t}g(\sigma)\dot{l}(\sigma)\,{\rm d}\sigma\notag\\
		&-\int_{0}^{t}\int_{-1}^{1}v(\sigma,y)\dot{\psi} (\sigma, y)\,{\rm d}y\,{\rm d}\sigma
		+\int_{0}^{t}\int_{-1}^{1} v_y(\sigma, y)\psi_y(\sigma,y)\,{\rm d}y\,{\rm d}\sigma\notag\\
		&-\frac12\int_{0}^{t}\int_{-1}^{1}v^2(\sigma ,y)\psi_y(\sigma,y)\,{\rm d}y\,{\rm d}\sigma
		=\int_{0}^{t}u(\sigma)l(\sigma)\,{\rm d}\sigma,
	\end{align}
	for every $t\in [0,T]$ and for every
	\begin{align}
		&\left[\begin{gathered}
			\psi\\
			l
		\end{gathered}\right]\in \left[ H^1\left((0,T);\;L^2(-1,1)\right)\cap  L^2\left((0,T);\;H^1_0(-1,1)\right)\right]\times H^1(0,T),\\
		&l(t)=\psi(t,h(t))\qquad\qquad\qquad \left(t\in[0,T]\right).
	\end{align}
A triple $\begin{bmatrix}
			v\\
			h\\
			g
		\end{bmatrix}$ which is a finite energy solution on $[0,T]$ for every $T>0$ is called a global finite energy solution.
\end{definition}
Note that the test functions used above depend on the solution and more precisely on its component $h$. Here and in the rest of this paper, for each $m\in \mathbb{N}$, $H^m$ and $H^m_0$ denote the classical Sobolev spaces and $H^{-m}$ stands for the dual of $H^m_0$ with respect to the pivot space $L^2$.

One of the contributions of our work is that we prove the existence and uniqueness of global
finite energy solution of \eqref{conversition_of_mom_of_solid_without_control}-\eqref{inlitial_date_of_solid_without_control} when $u$ is given by a feedback law of the form
\begin{equation}\label{control_design}
	u(t)=\mathcal{K}(h_1-h(t)) \qquad \qquad (\mathcal{K}\geq 0,\qquad t>0),
\end{equation}
with a given $h_1\in (-1,1)$ and $\mathcal{K}\geqslant 0$. This feedback may be regarded as a proportional
derivative (PD) controller, as is often used in control engineering. Another interpretation is that
the particle is actuated by a force which is generated by a spring (with elastic constant $\mathcal{K}$ and fixed at $h_1$). Unlike in \cite{tucsnak2015particle}, in Theorem \ref{global_exists}
below we do not assume that the initial velocity field $v_0$ and the initial data $g_0$ and $h_1-h_0$ are
small, in a suitable sense.

The first main result of this paper is:

\begin{theorem}\label{global_exists}
Let $v_0\in L^2(-1,1)$, $g_0\in\mathbb{R}$, $h_0,\ h_1\in(-1,1)$ and $\mathcal{K}\geqslant 0$. Then system \eqref{conversition_of_mom_of_solid_without_control}-\eqref{inlitial_date_of_solid_without_control}, with $u$ given by the feedback law  \eqref{control_design}, admits a unique global finite energy solution. Moreover, there exists $\alpha>0$ depending on $\Vert v_0\Vert_{L^2(-1,1)}$, $g_0$ $h_0$, $h_1$ and $\mathcal{K}$, such that the trajectory $h$ of the particle satisfies
			\begin{align}\label{estimateh}
			-1+\alpha
			\leq h(t)
			\leq 1-\alpha \qquad\qquad\qquad(t\geqslant 0).
		\end{align}
	Finally, the map defined in 
	\begin{equation}
		\begin{bmatrix} v_0\\ g_0\\ h_0\end{bmatrix}\mapsto
		\begin{bmatrix}  v\\  g\\  h\end{bmatrix},
	\end{equation} 
is continuous from
	$L^2(-1,1)\times\mathbb{R}\times (-1,1)$ to
	$\left[{C}([0,T];L^2(-1,1))\cap
	L^2((0,T);{H}_0^1(-1,1)) \right]\times L^2(0,T)\times
	{H}^1(0,T)$ for every $T>0$.
\end{theorem}

\begin{remark}\label{rem_comp}{\rm
The fact that $h(t)\in (-1,1)$ for evert $t\geqslant 0$ is not surprising: this has been shown, in a slightly different
context and for smoother initial data in \cite{Vzquez2006LACKOC}. The novelty we bring in Theorem
\ref{global_exists} is that we show that the distance from} $h(t)$ {\rm to} $\{-1,1\}$ {\rm has a positive lower bound}.
\end{remark}

\newpage

\noindent Our second main result describes the large time behavior of the solutions obtained in Theorem \ref{global_exists}.

\begin{theorem}\label{global_exp}
With the notation and assumptions of Theorem \ref{global_exists} we have:
\begin{itemize}
	\item If $\mathcal{K}=0$ the finite energy solutions
of \eqref{conversition_of_mom_of_solid_without_control}-\eqref{inlitial_date_of_solid_without_control}
satisfies
	\begin{align}\label{exp_velocity}
		&\Vert v(t,\cdot)\Vert^2_{L^2(-1,1)}+g^2(t)
		\notag\\
		\leq &\exp\left(-\frac14 t\right)\left(	\Vert v_0\Vert^2_{L^2(-1,1)}+g^2_0\right)\qquad (t>0).
	\end{align}
Moreover, there exists $h^*\in (-1,1)$ such that 
\begin{equation}
	\vert h(t)-h^*\vert^2\leq \exp\left(-\frac14 t\right)\left(	\Vert v_0\Vert^2_{L^2(-1,1)}+g^2_0\right).
\end{equation}
\item If $\mathcal{K}>0$ there exists a constant $\eta_2 > 0$, depending  on $\Vert v_0 \Vert_{L^2(-1,1)}$, $g_0$, $h_0$, $h_1$ and $\mathcal{K}$, such that
\begin{align}\label{input_states_stability}
	&\Vert v(t,\cdot)\Vert^2_{L^2(-1,1)}+g^2(t)+\mathcal{K}(h(t)-h_1)^2
	\notag\\
	\leq &16\exp\left(-\eta_2 t\right)\left(	\Vert v_0\Vert^2_{L^2(-1,1)}+g^2_0+\mathcal{K}(h_0-h_1)^2\right)\qquad (t>0).
\end{align}
\end{itemize}
	
\end{theorem}

The aforementioned theorems show, in particular, that when the feedback control \eqref{control_design} is active (that is $\mathcal{K}>0$)
the position of the particle tends to the prescribed ``target''  $h_1$, whereas when $\mathcal{K}=0$ it tends to an ``undetermined'' position $h^*$. By ``undetermined'' we mean that we are unable to describe $h^*$ in terms of the initial data.

The challenge of addressing the well-posedness question in this paper stems from two primary aspects. Firstly, the regularity of finite energy solutions necessitates avoiding the methodologies employed in \cite{vazquez2003large} and \cite{Vzquez2006LACKOC}. This constraint arises due to the specific properties of our solutions.
Secondly, our initial data and the parameter $\mathcal{K}$ are arbitrary, differing significantly from the setups in \cite{tucsnak2015particle} and \cite{Takahashi2015stabilization}. Consequently,  the lack of collisions between the particle and the boundary cannot be obtained just by proving that the solution remains small in an appropriate sense, as this has been done in \cite{tucsnak2015particle}.
Instead, we use a special test function introduced in Hesla \cite{Helsa2004PhD75H}.

An important novelty in Theorem \ref{global_exp} is that, unlike in \cite[Section 6]{tucsnak2015particle}, where it has been shown  that $\lim_{t\to \infty}h(t)=h_1$ we obtain  decay rate for $h_1 - h(t)$. An important ingredient in the proof of Theorem \ref{global_exp} is again the special test function mentioned above, which is used to construct a perturbed Lyapunov function.

The structure of this paper is as follows: In Section \ref{section_global} we prove the global existence and uniqueness result in Theorem \ref{global_exists}. Subsequently, in Section \ref{Global-expentional}, we give the proof of our global exponential stability result in  Theorem \ref{global_exp}.



\section{Global well-posedness of closed loop system}\label{section_global}
We begin by recalling a local in time existence and uniqueness result which is proved in \cite[Sections 2, 3 and 4]{tucsnak2015particle}

 \begin{proposition}\label{local_exists}
 	For every $\varepsilon>0$ and $\kappa>0$, we denote by $\mathcal{B}_{\varepsilon,\kappa}$ the set of
 $$\begin{bmatrix}
 		v_0\\
 		g_0\\
 		h_0\\
 	\end{bmatrix}\in L^2(-1,1)\times\mathbb{R}\times (-1,1)$$
  satisfying
 	\begin{align}
 		&\left\Vert v_0\right\Vert_{L^2(-1,1)}+g_0\leq \kappa,\\
 		&-1+4\varepsilon \leq h_0\leq 1-4\varepsilon.
 	\end{align}
 	Then there exists $T>0$, depending only on $\varepsilon$ and $\kappa$, such that for every $\begin{bmatrix}
 		v_0\\
 		g_0\\
 		h_0
 	\end{bmatrix}\in \mathcal{B}_{\varepsilon,\kappa}$, system \eqref{conversition_of_mom_of_solid_without_control}-\eqref{inlitial_date_of_solid_without_control} with feedback law \eqref{control_design} admits a unique finite energy solution (in the sense of Definition \ref{definition_of_finite_energy}) on the time interval $[0,T]$. Moreover, for every $t\in [0,T]$, we have the following energy estimate:
 \begin{align}\label{energy2}
 	&\int_{-1}^{1}v^2(t,y)\,{\rm d}y+g^2(t)+\mathcal{K}\left(h_1-h(t)\right)^2+2\int_{0}^{t}\int_{-1}^{1}v_y^2(\sigma, y)\,{\rm d}y\,{\rm d}\sigma\notag\\
 	= &\int_{-1}^{1}v_0^2(y)\,{\rm d}y+g^2_0+\mathcal{K}\left(h_1-h_0\right)^2,
 \end{align}
 \end{proposition}
\begin{remark}
	Let $\begin{bmatrix}  v\\  g\\  h\end{bmatrix}$ be the finite energy solution of system \eqref{conversition_of_mom_of_solid_without_control}-\eqref{inlitial_date_of_solid_without_control} on $[0,T]$, then the map
	\begin{equation}\label{continuity_new}
		\begin{bmatrix} v_0\\ g_0\\ h_0\end{bmatrix}\mapsto
		\begin{bmatrix}  v\\  g\\  h\end{bmatrix},
	\end{equation}
	is continuous from $\mathcal{B}_{\kappa,\varepsilon}$ to
	$\left\{\mathcal{C}([0,T];L^2(-1,1))\cap
	L^2((0,T);\mathcal{H}_0^1(-1,1)) \right\}\times L^2(0,T)\times
	\mathcal{H}^1(0,T)$.
\end{remark}
As a consequence of the above result we have:

\begin{corollary}\label{corollary_contraction}
	Assume that for every $\tau>0$, a finite energy solution of \eqref{conversition_of_mom_of_solid_without_control}-\eqref{inlitial_date_of_solid_without_control} defined on $[0,\tau)$
satisfies
	\begin{equation}
		\sup_{t \in [0,\tau)}\vert h(t)\vert <1.
	\end{equation}
Then the considered finite energy solution is global.
\end{corollary}

Proposition  below is the essential ingredient of our global well-posedness result:

\begin{proposition}\label{pro_global}
	Let $v_0\in L^2(-1,1)$, $g_0\in \mathbb{R}$, $h_0,\, h_1\in (-1,1)$.
Moreover, let $\tau>0$ and let $\begin{bmatrix}
 		v\\
 		g\\
 		h\\
 	\end{bmatrix}$
a finite energy solution of \eqref{conversition_of_mom_of_solid_without_control}-\eqref{inlitial_date_of_solid_without_control} defined on $[0,\tau)$.
 Then for every $t\in [0,\tau)$, the position of particle of system \eqref{conversition_of_mom_of_solid_without_control}-\eqref{inlitial_date_of_solid_without_control} with feedback law in \eqref{control_design} statisfies
	\begin{equation}\label{estimate_h}
		-1+\kappa_1(t) \leq h(t)\leq 1-\kappa_2(t) \qquad \qquad (t\in [0,\tau)),
	\end{equation}
where
\begin{align*}
	&\kappa_1(t)=\frac{2}{1+\max\left\{2,\frac{{1-h_0}}{1+h_0}\right\}\exp\left(C+2\mathcal{K}t\right)} &(& t\in [0,\tau)),\\ &\kappa_2(t)=\frac{2}{1+\max\left\{2,\frac{1+h_0}{1-h_0}\right\}\exp\left(\left(C+2\mathcal{K}t\right)\right)}    &(& t\in [0,\tau)),\\
	&C=10\left(\Vert v_0\Vert_{L^2(-1,1)}^2+g^2_0+\mathcal{K}(h_1-h_0)^2+\sqrt{\Vert v_0\Vert_{L^2(-1,1)}^2+g^2_0+\mathcal{K}(h_1-h_0)^2}\right).
\end{align*}
\end{proposition}
\begin{proof}
  Following the approach proposed in \cite[Section II]{Helsa2004PhD75H}, we begin by choosing the test function $\begin{bmatrix} \psi\\ l\end{bmatrix} =\begin{bmatrix}
		\varphi\\
		1
	\end{bmatrix}$ in the Definition \ref{definition_of_finite_energy},
where $\varphi$ is defined by:
\begin{equation}\label{testing_function}
	\varphi(t,y)=\begin{cases}
		\frac{1}{1+h(t)}(y+1) & \qquad\qquad(y\in (-1,h(t))),\\
		\frac{1}{1-h(t)}(1-y)                       & \qquad\qquad(y\in(h(t),1)).	
	\end{cases}
\end{equation}
It follows that
\begin{align}\label{intergrate_form}
	&\int_{-1}^{1}v(t,y)\varphi(t,y)\,{\rm d}y
	-
	\int_{-1}^{1}v_0(y)\varphi(0,y)\,{\rm d}y
	-
	\int_{0}^{t}\int_{-1}^{1}v(\sigma, y)\frac{\partial \varphi}{\partial \sigma}(\sigma, y)\,{\rm d}y\,{\rm d}\sigma\notag\\
	&+\int_{0}^{t} \int_{-1}^{1} v(\sigma, y) v_y(\sigma,y) \varphi(\sigma, y) \,{\rm d}y\,{\rm d}\sigma+g(t)-g_0+\int_{0}^{t}\int_{-1}^{1}v_y(\sigma, y)\varphi_y(\sigma, y)\notag\\
	=&\int_{0}^{t}\mathcal{K}(h_1-h(\sigma))\,{\rm d}\sigma,
\end{align}

Moreover, using the defintion of $\varphi$ in \eqref{testing_function} and the fact that $v(\sigma,h(\sigma))=\dot h(\sigma)$ for almost every $t\in [0, \tau),\;\sigma\in [0,\tau)$, we have
\begin{align}\label{caculate_v_xx}
	&\int_{-1}^{1}v_y(\sigma,y)\varphi_y(\sigma,y)\,{\rm d}y\notag\\
	=&\int_{-1}^{h(\sigma)}v_y(\sigma,y)\varphi_y(\sigma,y) \,{\rm d}y+\int_{h(\sigma)}^{1}v_y(\sigma,y)\varphi_y(\sigma,y)\,{\rm d}y\notag\\
	=&\frac{\dot{h}(\sigma)}{1+h(\sigma)}+\frac{\dot{h}(\sigma)}{1-h(\sigma)}\notag\\
	=&\frac{\rm d}{{\rm d} \sigma}(\ln (1+h(\sigma)))-\frac{\rm d}{{\rm d} \sigma}(\ln (1-h(\sigma))).
\end{align}

Using \eqref{defintion_of_P}, \eqref{defintion_of_A1}, \eqref{defintion_of_A2} and \eqref{caculate_v_xx}, we can rewrite \eqref{intergrate_form} as
\begin{align}\label{equality_of_h}
	&\ln \frac{1+h(t)}{1+h_0}-\ln \frac{1-h(t)}{1-h_0}\notag\\
	=&\int_{0}^{t}\mathcal{K}(h_1-h(\sigma))\,{\rm d}\sigma+P(0)-P(t)+\int_{0}^{t}A_1(\sigma ){\rm d}\sigma-\int_{0}^{t} A_2(\sigma )\,{\rm d}\sigma,
\end{align}
where, for almost every $t\in [0,\tau),\; \sigma \in [0,t]$, we have set
\begin{align}
	\label{defintion_of_P}&P(\sigma )=\int_{-1}^{1}\varphi(\sigma,y) v(\sigma,y)\,{\rm d}y+g(\sigma),\\
    \label{defintion_of_A1}&A_1(\sigma )=\int_{-1}^{1}v(\sigma,y)\frac{\partial \varphi}{\partial \sigma}(\sigma,y)\,{\rm d}y,\\
	\label{defintion_of_A2}&A_2(\sigma )=\int_{-1}^{1} v(\sigma,y)v_y(\sigma,y)\varphi(\sigma,y)\,{\rm d}y.
\end{align}
We use next the energy estimate \eqref{energy2} to estimate the four last terms in the right hand side of \eqref{equality_of_h}.

Firstly, we remark that for $P$ defined in \eqref{defintion_of_P} we can use \eqref{energy2}, the Cauchy-Schwarz inequality
and the fact that $|\varphi|$ is uniformly bounded by $1$  to get that
\begin{align}\label{estimate_P}
	\vert P (\sigma)\vert^2 \leq 4\left(\Vert v_0\Vert_{L^2(-1,1)}^2+g^2_0+\mathcal{K}\left(h_1-h_0\right)^2\right)\qquad\qquad (t\in [0,\tau),\ \sigma \in [0,t]).
\end{align}

To estimate $A_1$, we note that  from \eqref{defintion_of_A1} and the definition of $\varphi$ in \eqref{testing_function}, for $t\in [0,\tau))$ and almost every $\sigma\in [0,t]$ we have:
\begin{equation}\label{with_name}
	A_1(\sigma )=-\frac{g(\sigma)}{(1+h(\sigma))^2}\, \int_{-1}^{h(\sigma)}v(\sigma,y)(1+y)\, {\rm d}y
	+\frac{g(\sigma)}{(1-h(\sigma))^2}\, \int_{h(\sigma)}^{1}
	v(\sigma,y)(1-y)\, {\rm d}y.
\end{equation}
To estimate the first term the right hand side of the above formula, we firstly integrate by parts  to get
\begin{align}
	& \frac{g(\sigma)}{(1+h(\sigma))^2}\int_{-1}^{h(\sigma)}v(\sigma,y)(1+y)\,{\rm d}y \notag\\
	= &\frac{g(\sigma)}{2(1+h(\sigma))^2}\int_{-1}^{h(\sigma)}v(\sigma,y) \frac{\rm d}{{\rm d}y} \left[ (1+y)^2\right]
\, {\rm d}y\notag\\
	= &
	- \frac{g(\sigma)}{2(1+h(\sigma))^2}
	\int_{-1}^{h(\sigma)}v_y(\sigma,y) (1+y)^2\,{\rm d}y+
	\frac12 g^2(\sigma)
	\quad \left(t\in [0,\tau),\; \sigma \in [0,t] \quad {\rm a.e.}\right),
\end{align}
Moreover, by combining the Cauchy-Schwarz inequality, a trace theorem, the facts $\vert 1+h(\sigma)\vert \leqslant 2$ and $v(t,h(t))=g(t)$, we have
\begin{align*}
	&\left\vert
	 \frac{g(\sigma)}{2(1+h(\sigma))^2}
	\int_{-1}^{h(\sigma)} v_y(\sigma, y) (1+y)^2 \, \mathrm{d}y - \frac{1}{2}g^2(\sigma)
	\right\vert
	 \\
	 \leq {}&\frac{\vert g(\sigma) \vert }{2(1+h(\sigma))^2}
	 \left(
	 \int_{-1}^{h(\sigma)} v_y^2(\sigma, y) \, \mathrm{d}y
	 \right)^{\frac{1}{2}}
	\left(
	 \int_{-1}^{h(\sigma)} (1+y)^4 \, \mathrm{d}y
	 \right)^{\frac{1}{2}}
	 +\frac{1}{2}g^2(\sigma)\notag\\
	\leq{} &\frac{1}{\sqrt{5}}\vert g(\sigma)\vert (1+h(\sigma))^{\frac{1}{2}}
	\left( \int_{-1}^{h(\sigma)} v_y^2(\sigma, y) \, \mathrm{d}y \right)^{\frac{1}{2}} + \frac{1}{2}g^2(\sigma) \\
	\leq{} &g(\sigma) \left( \int_{-1}^{h(\sigma)} v_y^2(\sigma, y) \, \mathrm{d}y \right)^{\frac{1}{2}} + \frac{1}{2}g^2(\sigma) \quad  \\
	\leq{} &2g^2(\sigma) + 2 \int_{-1}^{h(\sigma)} v_y^2(\sigma, y) \, \mathrm{d}y  \\
	\leq{} &6\int_{-1}^{h(\sigma)} v_y^2(\sigma, y) \, \mathrm{d}y  \qquad\qquad\qquad(t\in [0,\tau),\ \sigma\in [0,t]\; {\rm a.e.}).
\end{align*}
Combining the above formula and \eqref{estimateA1} it follows that
\begin{align}\label{estimateA1}
&\left|\frac{g(\sigma)}{(1+h(\sigma))^2}\int_{-1}^{h(\sigma)}v(\sigma,y)(1+y)\,{\rm d}y\right|\notag\\
\leqslant& 6\int_{-1}^{h(\sigma)} v_y^2(\sigma, y) \, \mathrm{d}y  \qquad\qquad(t\in [0,\tau),\ \sigma\in [0,t]\; {\rm a.e.}).
\end{align}
In a completely similar manner we can estimate the second term in the right hand side of left term of \eqref{with_name}  to get
\begin{equation*}\label{A1_second}
\left|\frac{g(\sigma)}{(1-h(\sigma))^2}\int_{h(\sigma)}^1 v(\sigma,y)(1-y)\,{\rm d}y\right|
\leqslant 6\int_{h(\sigma)}^1 v_y^2(\sigma, y) \, \mathrm{d}y  \qquad\qquad(t\in [0,\tau),\ \sigma\in [0,t]\; {\rm a.e.}).
\end{equation*}
Putting together the last two inequalities and using \eqref{energy2}, it follows that
\begin{align}\label{estimate_of_A_1}
&\left\vert \int_{0}^{t}A_1(\sigma ){\rm d}\sigma\right\vert
\leq\int_{0}^{t}\vert A_1(\sigma )\vert {\rm d}\sigma\notag\\
\leq &6\int_{0}^{t}\int_{-1}^{1} v^2_y(\sigma,y)\,{\rm d}y\, {\rm d}\sigma
\leq 6 \left(\Vert v_0\Vert_{L^2(-1,1)}^2+g^2_0+\mathcal{K}\left(h_1-h_0\right)^2\right)\qquad \left(t\in [0,\tau)\right).
\end{align}
 We next estimate of $A_2$ defined in \eqref{defintion_of_A2}. Since $\varphi$ defined by \eqref{testing_function} is uniformly bounded by $1$, usingthe Cauchy-Schwarz  and Poincar{\'e} inequalities, it follows that for every $t\in [0,\tau)$ and almost every $\sigma\in [0,t]$, we have
\begin{align}\label{estimate_of_A2}
	&\vert A_2(\sigma)\vert
	\leq \int_{-1}^{1}\vert v(\sigma, y)v_y(\sigma,y)\phi(\sigma,y)\vert\, {\rm d}y\notag\\
	\leq& \left(\int_{-1}^{1} v^2(\sigma,y)\,{\rm d}y\right)^{\frac12}\left(\int_{-1}^{1} v_y^2(\sigma,y)\,{\rm d}y\right)^{\frac12}
	\leq4 \int_{-1}^{1}v^2_y(\sigma ,y)\,{\rm d}y.
\end{align}
From the above estimate and the energy estimate \eqref{energy2} we obtain:
\begin{align}\label{estimate_A2}
	&\left\vert \int_{0}^{t}A_2(\sigma ){\rm d}\sigma\right\vert
	\leq 4\int_{0}^{t}\int_{-1}^{1}v^2_y(\sigma ,y)\,{\rm d}y\,{\rm d}\sigma\notag\\
	\leq& 4\left(\Vert v_0\Vert_{L^2(-1,1)}+g^2_0+\mathcal{K}\left(h_1-h_0\right)^2\right) \qquad \qquad (t\in[0,\tau)).
\end{align}
Moreover, using the fact $\vert h(\sigma)-h_1\vert\leq 2$, for every $t\in (0,\tau),\;\sigma \in [0,t]$, it follows that
\begin{equation}\label{estimate_hu}
	\left\vert \int_{0}^{t}\mathcal{K}(h_1-h(\sigma)){\rm d}\sigma\right\vert \leq 2\mathcal{K}t.
\end{equation}
 By combing  \eqref{estimate_P}, \eqref{estimate_of_A_1}, \eqref{estimate_A2}, \eqref{estimate_hu} and \eqref{equality_of_h}, we have
\begin{align}
	\ln \frac{1+h(t)}{1+h_0}-\ln \frac{1-h(t)}{1-h_0}\leq C+2\mathcal{K}t \qquad \quad (t\in[0,\tau)),
\end{align}
where $C=10\left(\Vert v_0\Vert_{L^2(-1,1)}^2+g^2_0+\mathcal{K}(h_1-h_0)^2+\sqrt{\Vert v_0\Vert_{L^2(-1,1)}^2+g^2_0+\mathcal{K}(h_1-h_0)^2}\right)$.

After a simple caculation, the above inequality can be rephrased to
\begin{equation}\label{estimate_h_upper}
	h(t)\leq 1-\frac{2}{1+\max\left\{2,\frac{1+h_0}{1-h_0}\right\}\exp\left(C+2\mathcal{K}t\right)}\qquad\qquad (t\in[0,\tau)).
\end{equation}
Choosing next the test function $\begin{bmatrix}
	\psi\\
	l
\end{bmatrix}=\begin{bmatrix}
	-\varphi\\
	1
\end{bmatrix}$ in Defintion \ref{definition_of_finite_energy} and following step by step the procedure used to prove \eqref{estimate_h_upper}, it follows that:
\begin{equation}\label{estimate_h_lower}
h(t)\geq -1+\frac{2}{1+\max\left\{2,\frac{1-h_0}{1+h_0}\right\}\exp\left(C+2\mathcal{K}t\right)}   \qquad\qquad (t\in[0,\tau)).
\end{equation}
\end{proof}
 By combining  Corollary \ref{corollary_contraction} and Proposition \ref{pro_global} it follows that

 \begin{corollary}\label{cor_ex_glob}
 Let $\mathcal{K}\geqslant 0$. Then for every $v_0\in L^2(-1,1),\ g\in \mathbb{R}$ and $h_0,\ h_1\in (-1,1)$ the system \eqref{conversition_of_mom_of_solid_without_control}-\eqref{inlitial_date_of_solid_without_control} with feedback law \eqref{control_design} admits  unique global finite energy solution.
 \end{corollary}

We also need the  result below, which asserts that for $\mathcal{K}>0$ the equilibriul state $\begin{bmatrix} 0 \\0\\ h_1\end{bmatrix}$
is asymptotically stable.

 \begin{theorem}\label{thm:large_time_behavior}
 Let $\mathcal{K}> 0$. Then for every $v_0\in L^2(-1,1),\ g\in \mathbb{R}$ and $h_0,\ h_1\in (-1,1)$,
 the unique global finite energy solution of \eqref{conversition_of_mom_of_solid_without_control}-\eqref{inlitial_date_of_solid_without_control} with feedback law \eqref{control_design} satisfies
 	\begin{align}
		\lim_{t \to \infty} \|v(t, \cdot)\|_{L^2(-1,1)} = 0, \qquad
		\lim_{t \to \infty} g(t) = 0, \qquad
		\lim_{t \to \infty} h(t) = h_1.
	\end{align}
\end{theorem}

The proof of the above result can be done by closely following  the procedure in \cite[Section 6]{tucsnak2015particle}.
To be more precise, the result in \cite[Section 6]{tucsnak2015particle} required some smallness assumptions on the initial data,
but those assumptions were used there only to be sure that the solutions are global. Since now this fact is known from
Corollary \ref{cor_ex_glob}, the conclusion of Theorem \ref{thm:large_time_behavior} holds with no smallness assumption
on the initial data. However, for the sake of completeness, we provide a complete proof of Theorem \ref{thm:large_time_behavior} in the  Appendix in Section \ref{appendix}.

We are now in a position to prove the main result of this section.

 \begin{proof}[Proof of Theorem \ref{global_exists}]
 The existence and uniqueness of finite energy solutions has already been proven in
 Corollary \ref{cor_ex_glob}. We thus have only to prove \eqref{estimateh}.
 To this aim, we remark that
	\begin{enumerate}\label{estimate_K_equal_0}
		\item If $\mathcal{K} = 0$, from \eqref{estimate_h_upper} and \eqref{estimate_h_lower}, we have
		\begin{equation}\label{estimate_h_K_0}
			-1 + \frac{2}{1 + \max\left\{2, \frac{1 - h_0}{1 + h_0}\right\} \exp(C)} \leq h(t) \leq 1 - \frac{2}{1 + \max\left\{2, \frac{1 + h_0}{1 - h_0}\right\} \exp(C)},
		\end{equation}
		which directly implies estimate \eqref{estimateh}.
		
		\item If $\mathcal{K} > 0$, from Theorem~\ref{thm:large_time_behavior}, we have
		\begin{equation*}
			\lim_{t \to \infty} h(t) = h_1,
		\end{equation*}
so that there exists $T = T(h_1) > 0$ with
		\begin{equation}
			|h_1 - h(t)| \leq \frac{1}{8}(1 - h_1),
		\end{equation}
		which implies that for every $t > T(h_1)$,
		\begin{equation}
			1 - h(t) \geq \frac{7}{8}(1 - h_1).
		\end{equation}
		For $t \in (0, T(h_1))$, from \eqref{estimate_h_upper}, we have
		\begin{equation}
			h(t) \leq 1 - \frac{2}{1 +\max\left\{2,\frac{1 + h_0}{1 - h_0}\right \}  \exp(C + 2\mathcal{K}T(h_1))}.
		\end{equation}
		Therefore,
		\begin{equation}
			h(t) \leq 1 - \alpha,
		\end{equation}
		where
		\begin{align*}
			&C=10\left(\Vert v_0\Vert_{L^2(-1,1)}^2+g^2_0+\mathcal{K}(h_1-h_0)^2+\sqrt{\Vert v_0\Vert_{L^2(-1,1)}^2+g^2_0+\mathcal{K}(h_1-h_0)^2}\right),\\
			&\alpha = \min\left\{\frac{7}{8}(1 - h_1), \; \frac{2}{1 + \frac{1 - h_0}{1 + h_0} \exp(C + 2\mathcal{K}T(h_1))}\right\}.
		\end{align*}
		
		The lower bound of $h(t)$ can be obtained similarly, completing the proof.
	\end{enumerate}
\end{proof}

\section{Global exponential stability}\label{Global-expentional}

In this section, we give the proof of Theorem \ref{global_exp}.

\begin{proof}[Proof of Theorem \ref{global_exp}]
	 To begin let $\varepsilon$ be such that
\begin{equation}\label{pripa_eps}
0  \leq  \varepsilon\leq \min\left\{\frac18, \frac{\mathcal{K}}{8}\right\},
\end{equation}
and  we define
	\begin{align}\label{Lyapunov_functions_of_interaction_system}
		V_\varepsilon(t)
		=&	\int_{-1}^{1}v^2(t,y)\,{\rm d}y+g^2(t)+\mathcal{K}\left(h(t)-h_1\right)^2\notag\\
		&-\varepsilon\left(h_1-h(t)\right) P(t) \qquad\qquad\qquad \left(t\geqslant 0\right),
	\end{align}
	where  $\begin{bmatrix}
		v\\
		g\\
		h
	\end{bmatrix}$ is the global finite energy solution of system \eqref{conversition_of_mom_of_solid_without_control}-\eqref{inlitial_date_of_solid_without_control} with feedback \eqref{control_design}, which has been defined in Definition \ref{definition_of_finite_energy} and $P$ has been defined in \eqref{defintion_of_P}.

	Meanwhile, for every $t\geq 0$, we also define
	\begin{equation}\label{definition_of_energy}
		E(t)=\int_{-1}^{1}v^2(t,y)\,{\rm d}y+g^2(t)+\mathcal{K}\left(h(t)-h_1\right)^2\qquad (\mathcal{K}\geq 0).
	\end{equation}
	By combining  the Cauchy-Schwarz inequality and the fact $\vert\varphi\vert$ is uniform bounded by $1$, it is straightforward to verify that,
	\begin{equation}\label{norm_inequality2}
		\frac14 E(t)\leq V_{\varepsilon}(t)\leq 2E(t),
		\qquad
		\frac14 V_{\varepsilon}(t) \leq E(t)\leq 2 V_{\varepsilon}(t) \qquad\qquad (t\geq 0).
	\end{equation}
We next remark that $V_{\varepsilon}(t)$ is differentiable with respect to $t$.
Indeed, from the energy estimate given by \eqref{energy2} and Theorem \ref{global_exists}, we deduce that the mapping
\begin{equation}
	t \mapsto \frac{1}{2}\left(\int_{-1}^{1} v^2(t,y) \, \mathrm{d}y + g^2(t) + \mathcal{K}(h(t) - h_1)^2\right)
\end{equation}
is absolutely continuous and differentiable for almost every $t \in [0, \infty)$. Consequently, by taking the derivative of \eqref{definition_of_energy} with respect to $t$, we obtain
\begin{align}\label{energyestimatesection5}
	\frac{{\rm d}}{\rm d t}E(t)=-2\int_{0}^{1}v^2_y(t,y)\,{\rm d}y\qquad \qquad ( t\geq 0\;{\rm a.e.}).
\end{align}
On the other hand, from \eqref{equality_of_h} it follows that $P$ is absolutely continuous on every bounded interval contained in $[0,\infty)$,
thus differentiable almost everywhere on $[0,\infty)$. Moreover, from \eqref{defintion_of_P} it also follows that for almost every $t\geqslant 0$
we have
\begin{align}\label{derivative_expression}
	&\frac{\mathrm{d}P}{\mathrm{d}t}(t)
	= A_1(t) - A_2(t) - \int_{-1}^{1}v_y(t,y)\varphi_y(t,y) \, \mathrm{d}y + \mathcal{K}(h_1 - h(t))\notag\\
	=& A_1(t) - A_2(t) + \mathcal{K}(h_1 - h(t)) - \left(\frac{g(t)}{1+h(t)} + \frac{g(t)}{1-h(t)}\right),
\end{align}
where $A_1$ and $A_2$ have been defined in \eqref{defintion_of_A1} and \eqref{defintion_of_A2}, respectively.

 From \eqref{energyestimatesection5} and \eqref{derivative_expression} we get that $V_\varepsilon(t)$ defined by \eqref{Lyapunov_functions_of_interaction_system} is absolutely continuous, thus differentiable for almost every $t \in [0, \infty)$, with
 \begin{align}
 	\frac{\mathrm{d}}{\mathrm{d}t}V_\varepsilon(t)
 	=& -2\int_{-1}^{1} v_y^2(t,y) \, \mathrm{d}y
 	 + \varepsilon g(t) \left( g(t) + \int_{0}^{1} \varphi(t,y) v(t,y) \, \mathrm{d}y \right)  \notag \\
 	& - \varepsilon \left( h_1 - h(t) \right) \left( A_1(t) - A_2(t) + \mathcal{K}(h_1 - h(t)) - \left( \frac{g(t)}{1 + h(t)} + \frac{g(t)}{1 - h(t)} \right) \right).
 \end{align}
Moreover, we know from \eqref{estimateh} that there exists a constant $\alpha \in (0,1]$, such that
\begin{equation*}
	\min\{ 1-h(t), \;1+h(t)\}\geq \alpha\qquad (t\geq 0).
\end{equation*}
The above estimate, combined with an elementary  inequality, shows that for every $t\geq 0$ we have
\begin{equation*}
	\left(h_1-h(t)\right)\left( \frac{g(t)}{1 + h(t)} + \frac{g(t)}{1 - h(t)} \right)\leq \frac{g^2(t)}{\mathcal{K}\alpha^2}+\frac{\mathcal{K}}{4}(h_1-h(t))^2\qquad (\mathcal{K}>0).
\end{equation*}
From the above inequality it follows that for every $\varepsilon$ satisfying \eqref{pripa_eps}
and almosts every $t\geqslant 0$ and $\mathcal{K}>0$ we have
\begin{align*}
\frac{\mathrm{d}}{\mathrm{d}t}V_{\varepsilon}(t)
	\leq& \left(-\int_{-1}^{1}v_y^2(t,y)\,\mathrm{d}y - \varepsilon(h_1 - h(t))A_1(t)
+ \varepsilon(h_1 - h(t))A_2(t)\right)\\
&+ \left(5 + \frac{1}{\mathcal{K}\alpha^2}\right)\varepsilon g^2(t)
	- \frac{3\mathcal{K}\varepsilon}{4}(h_1 - h(t))^2.
\end{align*}

		Using the inequality above and the fact that $|h(t) - h_1| \leq 2$, along with the estimates of $A_1$ and $A_2$ in \eqref{estimateA1}, \eqref{estimate_A2}, respectively, we have
		\begin{align}\label{are_acuma}
			\frac{\mathrm{d}}{\mathrm{d}t}V_{\varepsilon}(t) \leq \left(-2 + 20\varepsilon\right)\int_{-1}^{1}v_y^2(t,y)\,\mathrm{d}y + \left(5 + \frac{1}{\mathcal{K}\alpha^2}\right)\varepsilon g^2(t) - \frac{3\mathcal{K}\varepsilon}{4}(h_1 - h(t))^2.
		\end{align}
Using the estimates
	\begin{align}\label{here_label}
		&g(t)=v(t,h(t))=\int_{-1}^{h(t)}v_y(t,y)\,{\rm d}y\leq \left(\int_{-1}^{1}v_y^2(t,y)\,{\rm d}y\right)^\frac{1}{2}\left(\int_{-1}^{1}1\,{\rm d}y\right)^\frac{1}{2}\notag\\
		\leq& \sqrt{2}\left(\int_{-1}^{1}v_y^2(t,y)\,{\rm d}y\right)^\frac{1}{2},
	\end{align}	
inequality \eqref{are_acuma} implies that for almost every $t\geqslant 0$ we have
\begin{align}
&\frac{\mathrm{d}}{\mathrm{d}t}V_{\varepsilon}(t) \leq \left(-2 + 20\varepsilon\right)\int_{-1}^{1}v_y^2(t,y)\,\mathrm{d}y + \left(5 + \frac{1}{\mathcal{K}\alpha^2}\right)\varepsilon g^2(t) - \frac{3\mathcal{K}\varepsilon}{4}(h_1 - h(t))^2\notag\\
=& \left(-2 + 20\varepsilon\right)\int_{-1}^{1}v_y^2(t,y)\,\mathrm{d}y + \left(7 + \frac{1}{\mathcal{K}\alpha^2}\right)\varepsilon g^2(t)-2\varepsilon g^2(t) - \frac{3\mathcal{K}\varepsilon}{4}(h_1 - h(t))^2\notag\\
\leq &\left(-1+20\varepsilon+2\left(7 + \frac{1}{\mathcal{K}\alpha^2}\right)\varepsilon\right)\int_{-1}^{1}v_y^2(t,y)\,{\rm d}y-2\varepsilon g^2(t)- \frac{3\mathcal{K}\varepsilon}{4}(h_1 - h(t))^2.
\end{align}
		We choose $\varepsilon = \frac{1}{16\left(34 + \frac{2}{\mathcal{K}\alpha^2}\right)}$, which satisfies  \eqref{pripa_eps}. Combining this with the Poincaré inequality, we obtain
		\begin{align}
			\frac{\mathrm{d}}{\mathrm{d}t}V_{\varepsilon}(t) \leq -\frac{1}{4}\left[1 - \left(34 + \frac{2}{\mathcal{K}\alpha^2}\right)\varepsilon\right]\int_{-1}^{1}v^2(t,y)\,\mathrm{d}y - 2\varepsilon g^2(t) - \frac{3\mathcal{K}\varepsilon}{4}(h_1 - h(t))^2.
		\end{align}
		
		Denote
		\begin{equation*}
			\eta = \frac{1}{4} \min\left\{ \frac{1}{34 + \frac{2}{\mathcal{K}\alpha^2}}, \; \frac{3\mathcal{K}\varepsilon}{4} \right\}.
		\end{equation*}
		
		Using the above inequality and \eqref{norm_inequality2}, for almost every $t \geq 0$, we have
		\begin{equation}
			\frac{\mathrm{d}}{\mathrm{d}t}V_{\varepsilon}(t) \leq -4\eta E(t) \leq -\eta V_{\varepsilon}(t),
		\end{equation}
		where $E(t)$ is defined by \eqref{definition_of_energy}.
		
		By Gronwall's inequality, for $\mathcal{K} > 0$, we obtain
		\begin{equation}
			V_{\varepsilon}(t) \leq \exp(-\eta t)V_{\varepsilon}(0).
		\end{equation}
		
		Using \eqref{norm_inequality2} again, we conclude that if $\mathcal{K}>0$ then
		\begin{equation}\label{energy_decay_K_positive}
			E(t) \leq 16\exp(-\eta t)E(0) \qquad\qquad (t \geqslant 0).
		\end{equation}
		
		For $\mathcal{K} = 0$, using \eqref{energyestimatesection5} and \eqref{here_label} for almost every $t\geq 0$, we have
		\begin{equation}\label{energy_K_0}
			\frac{\mathrm{d}}{\mathrm{d}t}E(t) = -2\int_{0}^{1}v_y^2(t,y)\,\mathrm{d}y \leq -\int_{0}^{1}v_y^2(t,y)\,\mathrm{d}y -\frac12 g^2(t).
		\end{equation}
		
		Applying the Poincaré inequality,  the above estimate implies that
		\begin{equation}
			\frac{\mathrm{d}}{\mathrm{d}t}E(t) \leq -\frac14E(t)\qquad \qquad (t\geq 0)
		\end{equation}
		
		By Gronwall's inequality, it follows that
		\begin{equation*}
			E(t) \leq \exp\left(-\frac14t\right)E(0), \qquad (t \geq 0).
		\end{equation*}
		The above estimate implies that
		\begin{equation}\label{estimate_g}
			\vert g(t)\vert\leq \exp\left(-\frac18 t\right)\sqrt{E(0)}.
		\end{equation}
 Next for any $\varepsilon>0$, we choose $T(\varepsilon)=16\ln \frac{\sqrt{E_0}}{4\varepsilon}$, then, for every $t_2>t_1\geq T(\varepsilon)$, using Newton-Leibniz formula we have
	\begin{align*}
		&\vert h(t_2)-h(t_1)\vert=\left\vert \int_{t_2}^{t_1}g(s)\,{\rm d}s\right\vert \leq \int_{t_2}^{t_1}\vert g(s)\vert \,{\rm d}s\notag\\
		\leq&\int_{t_2}^{t_1}\exp\left(-\frac18 t\right)\sqrt{E(0)} \,{\rm d}s\leq \frac{\sqrt{E_0}}{8}\left(\exp\left(-\frac18 t_1\right)-\exp\left(-\frac18 t_2\right)\right)\notag\\
		\leq &\frac{\sqrt{E_0}}{8}\exp\left(-\frac18 T(\varepsilon)\right)<\varepsilon,
	\end{align*}
by combining the above with \eqref{estimate_h_K_0}, there exists $h^*\in (-1,1)$, such that
\begin{equation}
	\lim_{t\to \infty}h(t)=h^*.
\end{equation}
Moreover, using \eqref{estimate_g} and above equality, we have
\begin{equation*}
	\vert h(t)-h^*\vert \leq \int_{t}^{\infty}\vert g(s)\vert\, {\rm d}s\leq \frac{\sqrt{E_0}}{8}\exp\left(-\frac18 t\right).
\end{equation*}
 \eqref{energy_decay_K_positive}, \eqref{energy_K_0} and the above inequality allow us to end the proof of Theorem \ref{global_exp}.
	\end{proof}

{\bf Acknowledgments.}
Funded by the European Union (Horizon Europe MSCA project Modconflex, grant number 101073558).
\section{Appendix}\label{appendix}
This appendix is devoted to the proof of Theorem \ref{thm:large_time_behavior}. This proof is essentially the same as the one proposed
in \cite{tucsnak2015particle} for a weaker version of this result, where it was assumed that the initial data satisfy some smallness conditions. Nevertheless, in order to convince the reader that the result holds without these assumptions, we give the proof below.

We first introduce the functions $W_1,\ W_2:L^2(-1,1) \times \mathbb{R}
\times (-1,1)\to [0,\infty)$ defined  by
\begin{align}&\label{LYAP1}
	W_1\begin{bmatrix}  \varphi \\ g\\
		h\end{bmatrix}=\frac12\left(\int_{-1}^{1}\varphi^2 \,{\rm d}y+
	|g|^2\right) &(&\varphi\in L^2(-1,1), \ g\in\mathbb{R},\ h\in
	(-1,1)),\\
&\label{LYAP2}
	W_2 \begin{bmatrix}  \varphi \\ g\\ h\end{bmatrix}=\frac{\mathcal{K}}{2}
	|h-h_1|^2 &(&\varphi\in L^2(-1,1), \ g\in\mathbb{R},\ h\in
	(-1,1)).
\end{align}
We also consider $D : \mathcal{H}_0^1(-1,1) \times
\mathbb{R} \times (-1,1)\to [0,\infty)$ defined  by
\begin{equation}\label{LYAP_DISIP}
	D \begin{bmatrix}  \varphi \\ g\\ h\end{bmatrix} =
	\int_{-1}^{1}\varphi_y^2(y)\, {\rm d}y  \qquad(\varphi\in
	{H}_0^1(-1,1), \ g\in\mathbb{R},\ h\in (-1,1)).
\end{equation}
Moreover, for $v_0\in L^2(-1,1)$, $g_0\in \mathbb{R}$ and
$h_0\in (-1,1)$ we set
\begin{equation}\label{notatie_desteapta}
	S(t)\begin{bmatrix} v_0\\ g_0\\ h_0\end{bmatrix} = \begin{bmatrix} v(t,\cdot)\\ g(t)\\ h(t)\end{bmatrix} \qquad(t\geqslant 0),
\end{equation}
where $\begin{bmatrix} v\\ g\\ h\end{bmatrix}$ is the corresponding finite energy solution of \eqref{conversition_of_mom_of_solid_without_control}-\eqref{inlitial_date_of_solid_without_control} constructed in Theorem \ref{global_exists}.
From Theorem \ref{global_exists} it follows that
to say that
\begin{equation}\label{continuity_very_new}
	S(t) \in C\left(L^2(-1,1)\times \mathbb{R}\times (-1,1);\;L^2(-1,1)\times \mathbb{R}\times (-1,1)\right) \qquad(t\geqslant 0),
\end{equation}
and that, for every $T>0$, the map
\begin{equation}\label{alta_timpenie}
	\begin{bmatrix}v_0\\ g_0\\ h_0\end{bmatrix}\mapsto D\left(S(\cdot)\begin{bmatrix}v_0\\ g_0\\ h_0\end{bmatrix}\right)
\end{equation}
is continuous from $L^2(-1,1)\times \mathbb{R}\times (-1,1)$ to $L^1(0,T)$.

With the above notation, the energy estimate \eqref{energy2}
can be written:
\begin{align}\label{ENCUW}
	&W_1\begin{bmatrix} v_0\\ g_0\\ h_0\end{bmatrix}+W_2\begin{bmatrix}
		v_0\\ g_0\\ h_0\end{bmatrix}-W_1\left(S(t)\begin{bmatrix} v_0\\
		g_0\\ h_0\end{bmatrix}\right)-W_2\left(S(t)\begin{bmatrix} v_0\\
		g_0\\ h_0\end{bmatrix}\right)\notag\\
	=&\int_0^t
	D\left(S(\sigma)\begin{bmatrix} v_0\\ g_0\\
		h_0\end{bmatrix}\right)\, {\rm d}\sigma  \qquad\qquad\qquad(t\geqslant 0).
\end{align}

\begin{proposition}\label{LYAP_PRIMA}
	Under the assumptions of Theorem \ref{global_exists} and with the above notation, for every $v_0\in L^2(-1,1)$, $g_0\in
	\mathbb{R}$ and $h_0\in (-1,1)$  we have
	$$
	\lim_{t\to\infty} W_1\left(S(t)\begin{bmatrix} v_0\\ g_0\\
		h_0\end{bmatrix}\right) = 0.
	$$
\end{proposition}

\begin{proof}
	Let
	$$
	W_k\left(S(t)\begin{bmatrix} v_0\\ g_0\\
		h_0\end{bmatrix}\right)=W_k(t) \qquad(k\in\{1,2\},\ \ t\geqslant
	0).
	$$
We use a contradiction argument.
	
If  $W_1$ does not converge to zero when $t\to\infty$.
	then there exist $\varepsilon>0$ and a sequence $(t_n)_{n\geq 0}$ of positive numbers with $t_n\to\infty$
	and
	$$
	W_1(t_n)\geqslant \varepsilon \qquad(n\in\mathbb{N}).
	$$
	Denote
	$$
	\delta_n=\max\left\{\delta>0\ \ |\ \ W_1(t_n-\delta)\geqslant
	\frac{\varepsilon}{2}\right\} \qquad\qquad\qquad(n\in\mathbb{N}).
	$$
	Since, according to \eqref{ENCUW} and to Poincar\'e's inequality and a trace theorem,
	we have that $W_1\in L^1[0,\infty)$, it follows that $$
	\sum_{n\in\mathbb{N}}\delta_n <\infty,
	$$
	so that
	$$
	\lim_{n\to\infty} \delta_n=0.
	$$
	On the other hand, we know from \eqref{ENCUW}  that $W_1+W_2$ is
	nonincreasing, so that
	\begin{align*}
		&\frac\varepsilon2+\frac12 |h (t_n-\delta_n)-h_1|^2=W_1(t_n-\delta_n)+W_2(t_n-\delta_n)\notag \\
		\geqslant& W_1(t_n)+W_2(t_n) \geqslant \varepsilon+\frac12 |h
		(t_n)-h_1|^2 \qquad\qquad(n\in\mathbb{N}).
	\end{align*}
	From the above estimate it follows that
	$$
	|h (t_n-\delta_n)-h_1|^2-|h (t_n)-h_1|^2 \geqslant \varepsilon
	\qquad\qquad\qquad(n\in\mathbb{N}).
	$$
	By combining the mean value theorem with the obvious inequality
	$$
	|h(t)-h_1|\leqslant 2 \qquad(t\geqslant 0),
	$$
	we obtain that for every $n\in\mathbb{N}$ there exist
	$\alpha_n\in (0,1)$ with
	\begin{equation}\label{contradictia}
	|\dot h(t_n-\alpha_n\delta_n)|\geqslant
	\frac{\varepsilon}{4\delta_n}\to \infty.
	\end{equation}
On the other hand, according to \eqref{ENCUW} we have that
$$
W_1(t)\leqslant W_1(0)+W_2(0) \qquad\qquad\qquad(t\geqslant 0.
$$
The above estimate and \eqref{LYAP1} clearly contradict \eqref{contradictia},
which ends the proof.
\end{proof}

We are now in a position to give the main  proof of this section.

\begin{proof}[Proof of Theorem \ref{thm:large_time_behavior}]
	 We first remark that Proposition \ref{LYAP_PRIMA} implies that
	\begin{equation}\label{un_cuplu}
		\lim_{t\to \infty} \|v(t,\cdot)\|_{L^2(-1,1)}=0, \qquad
		\lim_{t\to\infty} g(t)=0.
	\end{equation}
	On the other hand, the fact that $h(t)\in (-1,1)$ for every
	$t\geqslant 0$ implies that the set $(h(t))_{t\geqslant 0}$ is
	relatively compact in $\mathbb{R}$. Consequently, there exists a sequence
$(t_n)_{n\geqslant 0}$ of positive numbers with
	\begin{equation}\label{singurica}
		t_n\to \infty,\qquad \lim_{n\to\infty} h(t_n)=h^*\in
		[-1,1].
	\end{equation}
	Moreover, \eqref{ENCUW} implies that the map $t\mapsto
	D\left(S(t)\begin{bmatrix} v_0\\ g_0\\ h_0\end{bmatrix}\right)$ is
	in  $L^1[0,\infty)$. From this fact it follows  that for all $T>0$ we have
	$$
	\lim_{n\to\infty}\int_{t_n}^{T+t_n} D\left(S(t)\begin{bmatrix}
		v_0\\ g_0\\ h_0\end{bmatrix}\right)\, {\rm d}t=0.
	$$
	The above estimate can be combined with the semigroup property of the family
	$(S(t))_{t\geqslant 0}$ to imply that
	$$
	\lim_{n\to\infty}\int_{0}^{T} D\left(S(s)S(t_n)\begin{bmatrix}
		v_0\\ g_0\\ h_0\end{bmatrix}\right)\, {\rm d}s=0.
	$$
	On the other hand, we know from Theorem \ref{global_exists}
	that $S(s)$ is continuous on $L^2[-1,1]\times \mathbb{R} \times
	(-1,1)$, so that we can use \eqref{un_cuplu} and \eqref{singurica}
	to obtain
	\begin{equation}\label{should_be}
	\lim_{n\to\infty} S(t_n)\begin{bmatrix} v_0\\ g_0\\
		h_0\end{bmatrix}=
	\begin{bmatrix} 0\\ 0\\ h^*\end{bmatrix}.
	\end{equation}
	From the last two formulae and the fact that the map defined in \eqref{alta_timpenie}
is continuous we obtain that
	$$
	D\left(S(s)\begin{bmatrix} 0\\ 0\\ h^*\end{bmatrix}\right)=0
	\qquad(s\in [0,T]).
	$$
	To end the proof, we define
$$
\begin{bmatrix} \widetilde
		v(t,\cdot)\\ \widetilde g(t)\\ \widetilde h(t)\end{bmatrix}=S(t)
	\begin{bmatrix} 0\\ 0\\ h^*\end{bmatrix} \qquad\qquad\qquad(t\geqslant 0).
$$
 The from \eqref{should_be} it follows that
	$\widetilde v_y(s,\cdot)=0$ in $L^2(-1,1)$ for almost every $s\in
	[0,T]$. Moreover, since $\widetilde v$ vanishes for $y=\pm 1$, it
	follows that $\widetilde v(s,\cdot)=0$ in ${H}_0^1(-1,1)$
	for almost every $s\in [0,T]$. This implies, in particular, that
	$g(s)=0$ for almost every $s\in [0,T]$.
	Finally, using \eqref{newton_second_law_without_control} and \eqref{control_design} (with $\widetilde v$ instead of $v$ and $h^*$ instead of
	$h_0$) we obtain that $h^*=h_1$, which concludes the proof.
\end{proof}

\bibliographystyle{siam}
\bibliography{references}

\end{document}